\documentclass[11pt]{amsart}

\usepackage[centertags]{amsmath}
\usepackage{amsfonts}
\usepackage{amssymb}
\usepackage{amsthm}
\usepackage[english]{babel}
\usepackage[active]{srcltx}

\usepackage{pst-all}

\usepackage[colorlinks]{hyperref}

\newcommand{\R}{\mathbb{R}}

\newcommand{\N}{\mathbb{N}}

\newcommand{\Sx}{S_{X^{*}}}

\newtheorem{thm}{Theorem}[section]
\newtheorem{cor}[thm]{Corollary}
\newtheorem{lem}[thm]{Lemma}
\newtheorem{lema}[thm]{Lemma}
\newtheorem{prop}[thm]{Proposition}

\newtheorem{rem}[thm]{Remark}

\newtheorem{defn}[thm]{Definition}

\numberwithin{equation}{section}

\begin{document}


\title{A characterization of the Radon-Nikodym property}

\author[Robert Deville]{Robert Deville}
\address{Institut de Math\'ematiques de Bordeaux, Universit\'e Bordeaux 1, 33405, Talence, France}
\email{Robert.Deville@mat.u-bordeaux1.fr}

\author[\'Oscar Madiedo]{\'Oscar Madiedo}
\address{ Departamento de An\'alisis Matem\'atico, Universidad Complutense de
Madrid, 28040, Madrid, Spain}
\email{oscar.reynaldo@mat.ucm.es}

\thanks{Research supported in part by MICINN Project MTM2009-07848 (Spain).
The authors wish to thank the Institut de Math\'ematiques de Bordeaux
where this research has been carried out.
O.Madiedo is also supported by grant BES2010-031192.}


\date{Julio,  2012}
\maketitle

\begin{abstract}
It is well known that every bounded below and non increasing sequence 
in the real line converges. We give a version of this result valid in Banach spaces 
with the Radon-Nikodym property, thus extending a former result of A. Proch\'azka.
\end{abstract}


\section{ Introduction.}

Our purpose is to state an analogue of the fact that every bounded below and non increasing sequence 
in the real line $\R$ converges in the framework of a  Banach space $X$. This is not clear,
even whenever $X=\R^2$. However, we shall see that it is indeed possible
in Banach spaces with the Radon-Nikodym property.

\begin{defn}
Let X be a Banach space. We say that $X$ has the Radon-Nikodym property
if, for every non empty closed convex bounded subset $C$ of $X$ and every $\eta>0$,
there exists $g$ in the unit sphere of the dual of $X$ and $c\in\R$ such that  
$\{x\in C;\,g(x)<c\}$ is non empty and 
has diameter less than $\eta$. 
\end{defn}

Every reflexive Banach space has the Radon-Nikodym property, but $L^1([0,1])$
and $\mathcal C(K)$ spaces whenever $K$ is an infinite compact space fail this poperty.
Moreover, if $Y$ is a subspace of a Banach space with the Radon-Nikodym property,
then $Y$ has the Radon-Nikodym property. The Radon-Nikodym property
can be characterized in many ways, see \cite{RB}, \cite{DeMa} and \cite{RP}.

Before stating our main result, we need some notations. 
If $X$ is a real Banach space,
$S_X$ stands for its unit sphere and $\Sx$ for the unit sphere of its dual. For $f \in X^*$ 
and $r > 0$ we denote 
$\overline{B} (f, r) = \{g \in X^* : \|f - g\| \leqslant r\}$ 
and $B(f,r) = \{g \in X^* : \|f - g\| < r\}$ the closed and open ball centered at $f$ and of radius $r$ respectively. 
Let us recall that whenever $X$ is a Banach space, $g\in X^*$ and $c\in\R$, 
we denote $\{g\geqslant c\}$
the closed half space $\{u\in X;\,g(u)\geqslant c\}$ and $\{g<c\}$ 
the open half space $\{u\in X;\,g(u)<c\}$.
If $C$ is a non empty convex subset of $X$, the set $C\cap\{g\geqslant c\}$ 
is called a closed slice of $C$
and $D\cap\{g<c\}$ an open slice of $C$.
If $x\in X$ and $f\in X^*$, we shall use both notations $y(x)$ and $\langle f,x\rangle$
for the evaluation of $f$ at~$x$.

\begin{thm}\label{converges}
Let $X$ be a Banach space with the Radon-Nikodym property.
Let $f \in \Sx$ and  $\varepsilon \in(0,1)$ be fixed. There exists a function 
$t: X \to \Sx \cap B(f,\varepsilon)$ such that for all sequence $(x_n)$, 
if the sequence $\bigl(f(x_n) - \varepsilon \|x_n\|\bigr)$ is bounded below and if  
$\langle t(x_n), x_{n+1} - x_n \rangle \leqslant 0$ for all $n \in \N$, then the sequence $(x_n)$ 
converges in $X$.
\end{thm}

\begin{rem} \rm
Theorem \ref{converges} can be reformulated in terms of games. 
This presentation was introduced in \cite{MZ}, see also \cite{DeMa} and \cite{Z}. There are two players
$A$ and $B$ who play alternatively. 
Player $A$ chooses linear functionals $f_n\in\Sx\cap B(f,\varepsilon)$ 
and player $B$ chooses  
$x_n$ in the cone $\{x\in X; f(x)-\varepsilon\|x\|+m\geqslant 0\}$ for some $m\in\R$,
with the following rules.
\begin{itemize}
\item[-] player $B$ chooses a point $x_0$;
\item[-] once  $B$ has played $x_n$, $A$ chooses $f_n\in\Sx\cap B(f,\varepsilon)$;
\item[-] once $A$ has played $f_n$, $B$ chooses $x_{n+1}$ such that $f_n(x_{n+1}-x_n)\leqslant 0$.
\end{itemize}
Player $A$ wins if the sequence $(x_n)$ converges. A winning tactic for player $A$
is a function $t: X \to \Sx \cap B(f,\varepsilon)$ such that, if for each $n$,
$f_n=t(x_n)$, then $A$ wins the game. Theorem \ref{converges} expresses the fact that in spaces with the Radon-Nikodym property,
player $A$ has always a winning tactic.
\end{rem}

Let us give a particular case of Theorem \ref{converges}. We assume here that $X=\R^2$, which has the Radon-Nikodym property.
It is clear that if $(x_n,y_n)$ is a sequence in $\R^2$ such that $(y_n)$ is non increasing and bounded below,
then the sequence $(y_n)$ converges, but in general the sequence $(x_n,y_n)$ does not converge, even if we require that
the sequence $(x_n,y_n)$ is included in a cone $C=\{(x,y); y-\varepsilon|x|+m\geqslant 0\}$ 
for some $\varepsilon>0$ and $m\in\R$. 
An obvious consequence of our Theorem is~:

\begin{cor}
Given $0<\varepsilon<1/2$, there exists a function $\tau:\R^2\to]-\varepsilon,\varepsilon[$ such that for every
sequence $(x_n,y_n)\in\R^2$, if
the sequence $(y_n-\varepsilon|x_n|)$ is bounded below and if  
$y_{n+1}-y_n\leqslant \tau(x_n,y_n)(x_{n+1}-x_n)$ for all $n\in\N$,
then the sequence $(x_n,y_n)$ converges. 
\end{cor}

\begin{proof}
Assume that $X=\R^2$ is endowed with the norm $\|(x,y)\|_1=|x|+|y|$ and that $0<\varepsilon<1$. Fix $f\in X^*$ with coordinates $(0,1)$.
Observe first that  if $X_n\in\R^2$ has coordinates $(x_n,y_n)$ and if the sequence $(y_n-\varepsilon|x_n|)$ is bounded below,
then the sequence $\bigl(f(X_n)-\frac{\varepsilon}{1+\varepsilon}\|X_n\|_1\bigr)$ is bounded below.
Applying Theorem \ref{converges}, there exists $t:X\to \Sx\cap B(f,\frac{\varepsilon}{1+\varepsilon})$
such that if the sequence $\bigl(f(X_n)-\frac{\varepsilon}{1+\varepsilon}\|X_n\|\bigr)$ is bounded below and 
$\langle t(X_n),X_{n+1}-X_n\rangle\leqslant 0$ for all $n\in\N$,
then the sequence $(X_n)$ converges in $\R^2$.
On the other hand, $X^*$ is $\R^2$ endowed with the supremum norm. Since
$t(x,y)\in \Sx\cap B(f,\varepsilon)$, we have that the coordinates of $t(x,y)$ are  of the form $(-\tau(x,y),1)$, with $-\varepsilon<\tau(x,y)<\varepsilon$.
Finally,
the condition $\langle t(X_n),X_{n+1}-X_n\rangle\leqslant 0$
is equivalent to $y_{n+1}-y_n\leqslant \tau(X_n)(x_{n+1}-x_n)$.
\end{proof}

\begin{rem} \rm The above result is an improvement of the following result of A. Proch\'azka, see \cite[Therorem 2.3]{AP}.
\\
\sl Let $X$ be a Banach space with the Radon-Nikodym property and $K$ be a closed convex bounded subset of $X$.
There exists a function $t: K \to \Sx $ such that for all sequence $(x_n)$ in $K$, 
if  $\langle t(x_n), x_{n+1} -x_n \rangle \leqslant 0$ for all $n \in \N$, then the sequence $(x_n)$ 
converges in~$X$.
\rm
\\
 Theorem \ref{converges} extends the above result in three manners.
\begin{itemize}
\item[-]  The tactic $t$ is defined on all the space $X$.
\item[-]  The hypothesis that the sequence $(x_n)$ is bounded $(x_n\in K$)
is replaced by the weaker hypothesis the sequence
$\bigl(f(x_n) - \varepsilon \|x_n\|\bigr) $ is bounded below, which means that the sequence $(x_n)$ lies in a cone $\{x;\,f(x)-\varepsilon\|x\|+m\geqslant 0\}$ for some $m\in\R$.
\item[-]  The tactic $t$ in our theorem takes its values only in a subset of $\Sx$ of small diameter.
\end{itemize}
\end{rem}

\begin{rem}\rm
Let us notice that Theorem \ref{converges} is actually a characterization of the Radon-Nikodym property.
Indeed, if $X$ fails the Radon-Nikodym property, there exists a non empty convex bounded subset $C$ of $X$ and $\eta>0$,
such that for all $f\in\Sx$ and $c\in\R$, if the slice $C\cap\{f<c\}$ is non empty, then it  
has diameter greater than $2\eta$. Moreover, we can assume that $C$ is open. 
Indeed, if $\delta<\eta$, the set $C+B(0,\delta)$ is open and all its slices have 
diameter greater than $2(\eta-\delta)$.
Now let  $(f_n)$ be a sequence in $\Sx$. We construct inductively a sequence $(x_n)$ in $C$ as follows. 
We choose arbitrarily $x_0\in C$. Once
$x_n$ has been constructed, we note that the slice $C\cap\{f_n< f_n(x_n)\}$ is non empty
because $x_n\in C$ and $C$ is open, so this slice has diameter 
greater than $2\eta$,
hence we can choose $x_{n+1}$ in $C$ such that $f_n(x_{n+1}-x_n)< 0$ and 
$\|x_{n+1}-x_n\|\geqslant\eta$.
Moreover, since $\{f(x)-\varepsilon\|x\|;\, x\in C\}$ is bounded below, we have in particular that $\{f(x_n)-\varepsilon\|x_n\|;\, n\in\N\}$ is bounded below.
This clearly contradicts the existence of a function $t$ with the property of Theorem \ref{converges}.
\end{rem}

\begin{rem}\rm
Let us notice particular cases of Theorem \ref{converges} have been obtained in \cite{MZ} and  \cite{DeMa},
and used there to give a simple proof of Buchzolich's solution of the Weil gradient problem, and also used
in \cite{DJ} to construct almost classical solutions of Hamilton-Jacobi equations. 
\end{rem}

Our paper is organized as follows. The following section is devoted to the proof of two 
elementary geometrical lemmas.
In section 3, we define a mapping $t$ on a given subset of $X$
such that for every sequence $(x_n)$ in this subset satisfying the assumptions of Theorem \ref{converges}, the sequence 
$(x_n)$ is $\eta$-Cauchy for some $\eta>0$. Such a mapping will be called $\eta$-tactic.
In the following section we prove that every mapping which is near (in some sense) the function $t$
is also an $\eta$-tactic. We are thus led to the definition of multi-$\eta$-tactic.
We then construct, for a given sequence $(\eta_k)$ tending to $0$, a decreasing sequence of multi-$\eta_k$-tactics,
and we prove finally in the last section Theorem \ref{converges}.

\section{Slices.}

The following lemma expresses the fact that if $D$ is a  closed convex set of $X$,
possibly unbounded, and if $S$ is a bounded slice defined by $\widehat{f} \in S_{X^*}$,
then functionals which are in a neighborhood of  $\widehat{f}$ define slices 
of $D$ included in $S$.

\begin{lema}\label{radius}
Let $D$ be a closed convex set of $X$, $\widehat{f} \in S_{X^*}$  and $c \in \R$. 
Assume that  $S = D \cap \{\widehat{f} < c\}$ 
is bounded and that both $S$ and $D\backslash S$ are non empty.
Let us denote $M:=\max\{\|u\|; u \in S\}$ and $R(x)= \frac{c-\widehat{f}(x)}{4M}$. 
If $x \in S$ and $g \in S_{X^*}$ satisfy $\|g-\widehat{f}\| \leqslant R(x)$, 
then $\bigl(D\backslash S\bigr)\cap\{g\leqslant g(x)\}=\emptyset$.
\end{lema}

\begin{proof}
It is clear that $0<M<+\infty$, so, for $x\in S$, $R(x)$ is well defined and $R(x)>0$. 
Let us assume  that $\bigl(D\backslash S\bigr)\cap\{g\leqslant g(x)\}\ne\emptyset$ 
and fix $z\in D\backslash S$ such that $g(z) \leqslant g(x)$.
There exists a unique $q\in[0,1]$ such that, if 
$y=qx+(1-q)z$, then $\widehat{f}(y)=c$.
Thus $y$ is in the closure of $S$ and $\| y\|\leqslant M$. 
On the other hand, by linearity of $g$, $g(z)\leqslant g(y)\leqslant g(x)$.
By hypothesis,  $g(x) \leqslant \widehat{f}(x) + R(x)\|x\|\leqslant \widehat{f}(x) + MR(x)$. 
Hence 
$$
\widehat{f}(y) \leqslant  g(y) +\|g-\widehat{f}\|\|y\| \leqslant g(x) + R(x)M \leqslant \widehat{f}(x) + 2MR(x)
= \frac{\widehat{f}(x) + c}{2} < c
$$
Thus $\widehat{f}(y) < c$.
This contradiction concludes the proof.

\end{proof}


\begin{center} 
\begin{picture}(140,100)(-70,-20) 
 \bezier{400}(-100,50)(-20,-40)(60,50)

\put(-90,20){\line(5,1){150}}

 \put(-40,42){$\{\widehat{f} = c\}$}

\put(-30,5){\line(5,2){80}}

 \put(6,13){$C$}
 \put(-92,45){$D$}
 \put(-30,18){$S$}
 
 \put(53,35){$\{g = g(x)\}$}
 \put(-9,12){$+$}
 \put(-5,18){$x$}

 \put(-35,-10){{\bf Figure 1}} 

\end{picture}
\end{center}


If $D$ is a closed convex set of a Banach space $X$ and if $g\in X^*$, 
we say that $g$ strongly exposes $D$ if $diam(D\cap\{g<c\})$ tends to $0$ as $c$ tends to 
$\inf\{g(u);\,u\in D\}$. 
The following lemma expresses the fact that if $D$ is a  closed convex set of 
a Banach space with the Radon-Nikodym property,
and if $S$ is a bounded slice defined by $\widehat{f} \in S_{X^*}$,
then there exists functionals in a neighborhood of  $\widehat{f}$ 
that define small slices of $D$ included in $S$.

\begin{lem}\label{diam} Assume that $X$ has the Radon-Nikodym property.
Let $\eta,r>0$ and $D$ be a closed convex set of $X$. Let $\widehat{f}\in X^*$ and $c\in\R$ be such that $S=D \cap \{ \widehat{f} < c\}$ is a non empty bounded set. 
Then, there exists $g \in \Sx$  and $d \in \R$ such that, 
if $C = D\cap\{g< d\}$, then 
\begin{itemize}
\item[(i)] $C \neq \emptyset$, $diam \;C < \eta$ and $C\subset S$,
\item[(ii)] $\big\|g - \widehat{f}\big\| <\min\big\{r, \inf\{R(u); u \in C\}\big\}$.
\end{itemize}
\end{lem}

\begin{proof}
We first claim that if $\tau\!>\!0$,
there exists $g_\tau\!\in\! X^*$ such that $\|g_\tau\!-\!\widehat{f}\|\!<\!\tau$ and 
$g_\tau$ strongly exposes $D$ at some point 
$x_\tau\in D\cap\{\widehat{f} < c\}$. 
Indeed, the set $\overline{S}=D\cap\{\widehat{f} \leqslant c\}$ is a nonempty closed convex bounded subset of $X$.
Thus, the set $\{g\in X^*;\,g\text{ strongly exposes }\overline{S}\}$ is dense in $X^*$(see \cite{RB}).
For each $\tau>0$, we select $g_\tau\in X^*$ and $x_\tau\in\overline{S}$
such that $\| \widehat{f}-g_\tau\|\leqslant \tau$ and $g_\tau$ strongly exposes 
$\overline{S}$ at $x_\tau$.
We shall now use the following~:

\smallskip\sl
Fact : $R(x_\tau)$ converges to 
$\sup\{R(u);\,u\!\in\! \overline{S}\}=\sup\{R(u);\,u\!\in\! D\}>0$ as $\tau$ goes to $0$.

\smallskip\noindent\rm
Since $R(x)=\gamma\bigl(c-\widehat{f}(x)\bigr)$ where $\gamma$ is a positive constant, it is enough to prove that 
$\widehat{f}(x_\tau)$ converges to $\inf\{\widehat{f}(x);\,x\in \overline{S}\}$. 
If we denote $A=\sup\{\|x\|;\,x\in \overline{S}\}$, we have
$$
\widehat{f}(x_\tau)\leqslant g_\tau(x_\tau)+A\|g_\tau-\widehat{f}\|\leqslant \tau A+g_\tau(x)
$$
for all $x\in \overline{S}$. Thus
$$
\widehat{f}(x_\tau)\leqslant \tau A+\widehat{f}(x)+\|\widehat{f}-g_\tau\|\cdot\|x\|\leqslant 2\tau A + \widehat{f}(x)
$$
Taking the infimum over all $x\in \overline{S}$, we obtain
$$
\inf\{\widehat{f}(x);\,x\in \overline{S}\}\leqslant \widehat{f}(x_\tau)\leqslant 2\tau A+\inf\{\widehat{f}(x);\,x\in \overline{S}\}
$$
and this proves the fact. Since $\sup\{R(u);\,u\!\in\! D\}>0$, if $\tau$ is small enough, we have  $R(x_\tau)>0$, thus
$g_\tau$ strongly exposes $\overline{S}$ at some point 
$x_\tau\in D\cap\{ \widehat{f}<c\}$, hence $g_\tau$ strongly exposes $D$ at some point 
$x_\tau\in D\cap\{\widehat{f}<c\}$, and this proves the claim.

We now prove the lemma. We fix $\tau$ such that $\tau\leqslant\min\{r,\sup\{R(u);\,u\in D\}/2\}$
and such that $R(x_\tau)>\sup\{R(u);\,u\in D\}/2$. 
Let us denote $C_\delta=D\cap\{g_\tau<g_\tau(x_\tau)+\delta\}$.
Using the continuity of $R$ and the 
fact that $g_\tau$ strongly exposes $D$ at $x_\tau$, we have that 
$\inf\{R(u);\,u\in C_\delta\}$ tends to $R(x_\tau)$. We now fix $\delta>0$ small enough
so that $\inf\{R(u);\,u\in C_\delta\}>\sup\{R(u);\,u\in D\}/2$
and $diam(C_\delta)<\eta$.
We now put $g=g_\tau$ and $d=g_\tau(x_\tau)+\delta$.
The set $C=C_\delta= D\cap\{g <  d\}$ is non empty and $diam(C)<\eta$.
Since $\inf\{R(u);\,u\in C\}>0$, we have that $C\subset S$.
Finally, $\| g-\widehat{f}\|<\tau\leqslant\min\{r,\sup\{R(u);\,u\in D\}/2\}\leqslant\min\{r,\inf\{R(u);\,u\in C\}\}$.
\end{proof}

\section{$\varepsilon$-tactics.}

We fix a Banach space $X$ with the Radon-Nikodym property, $f\in\Sx$ and $0<\varepsilon<1$.
For $p\in\mathbb{Z}$, we define $\Lambda_p = \{x; f(x) \geqslant \varepsilon \|x\|+p\}$.
For all $p$, $\Lambda_p$ is a closed convex unbounded subset of $X$,
$\Lambda_q\subset\Lambda_p$ whenever $p\leqslant q$,
$\Lambda_0$ is a  cone of $X$, and if $p\geqslant 0$, for all $x\in\Lambda_p$ and all $\tau\geqslant 1$, 
$\tau x\in\Lambda_p$.
The following result says that if $D$ is a convex set containing $\Lambda_{p+1}$, different from $\Lambda_{p+1}$,
and included in $\Lambda_p$,  then there exists a small slice of $D$ that does not intersect $\Lambda_{p+1}$.

\begin{lem}\label{diamC}
Let $\eta\!>\!0$, $p\in\mathbb{Z}$ and $D$ be a closed convex set of $X$ such that $\Lambda_{p+1}\subset D\subset \Lambda_p$ 
and $D\ne\Lambda_{p+1}$.
Then, there exists  $g\in X^*$, $\|g-f\|<\varepsilon$ and $d\in\R$ 
such that
$$
C=D\cap \{g <  d\}\ne\emptyset,
\qquad
C\cap\Lambda_{p+1}=\emptyset \quad\text{and}\quad diam\:( C )< \eta.
$$
\end{lem}


\begin{center} \begin{picture}(140,130)(-50,-30) 


\put(0,-15){\line(1,0){80}} 
\put(0,-15){\line(-1,0){70}}

\put(-22,-30){{\bf Figure 2}} 
\put(22,35){{$C$}} 
\put(-60,52){{$D$}} 

 \put(80,-12){$Ker f$}

\bezier{400}(-8,8)(0,2)(8,8) 
\put(8, 8){\line(3,2){70}}
\put(-8, 8){\line(-3,2){58}}

\bezier{500}(-8,40)(0,35)(8,40) 
\put(-8, 40){\line(-3,2){58}}
\put(8, 40){\line(3,2){70}}

\bezier{400}(-60,50)(5,10)(50,38) 
\put(70,51.3){\line(-3,-2){20}}

\put(3.3, 25){\line(3,1){58}}
\put(8.6, 26.5){\line(3,1){49}}
\put(9.5, 26.6){\line(3,1){47}}
\put(10.4, 26.7){\line(3,1){45}}
\put(11.28, 26.8){\line(3,1){43.2}}
\put(11.7, 26.6){\line(3,1){42.5}}
\put(12.4, 26.5){\line(3,1){41}}
\put(13.3, 26.5){\line(3,1){38.5}}
\put(14.1, 26.6){\line(3,1){37}}
\put(15, 26.5){\line(3,1){35.5}}
\put(17.5, 26.95){\line(3,1){32}}
\put(19.5, 27.3){\line(3,1){28}}
\put(21, 27.5){\line(3,1){25}}
\put(22.5, 27.6){\line(3,1){22}}
\put(25, 28.2){\line(3,1){16.5}}
\put(26.8, 28.4){\line(3,1){12}}
\put(80,90){$\Lambda_{p+1}$}
\put(80,50){$\Lambda_p$}
\end{picture}
\end{center}


\begin{proof}
Let us pick $x_0\in D\backslash\Lambda_{p+1}$. According to the Hahn-Banach Theorem,
there exists $h\in X^*$ such that 
\begin{equation}
h(x_0)<\inf\{h(x);\,x\in\Lambda_{p+1}\}.\label{1}
\end{equation}
Without loss of generality, we can assume that $\|h\|=1$.

\smallskip\noindent
Claim 1 :  $h(x)=0$ implies $f(x)\leqslant\varepsilon\|x\|$.

\noindent
Indeed, if $h(x)=0$, then, for all $\tau>0$, $h(\tau x+x_0)=h(x_0)$, hence,
according to inequality \eqref{1}, 
$f(\tau x+x_0)<\varepsilon\| \tau x+x_0\|+p+1\leqslant \tau\varepsilon\|x\|+\varepsilon\|x_0\|+p+1$.
On the other hand, $x_0\in\Lambda_p$, so $f(x_0)\geqslant\varepsilon\|x_0\|+p$,
and the above inequality implies
$$
f(x)\leqslant\varepsilon\|x\|+\frac 1 \tau
$$ 
The claim is proved since this is true for all $\tau>0$. 

\smallskip\noindent
Claim 2 :  there exists $\lambda>0$ such that $\| f-\lambda h\|\leqslant\varepsilon$.

\noindent
It follows from claim 1 and the Hahn-Banach theorem that there exists $h'\in X^*$ such that 
$\| h'\|=\varepsilon$ and for all $x\in Ker(h)$, $h'(x)=f(x)$.
Therefore, there exists $\lambda\in\R$ such that $f-h'=\lambda h$.
Pick $x\in \Lambda_1\cap\Lambda_{p+1}$. 
This implies that $\tau x\in\Lambda_{p+1}$ for all $\tau>1$. 
If $h(x)<0$, then $h(\tau x)$ tends to $-\infty$
as $\tau$ tends to $+\infty$, which contradicts the fact that $\tau x\in\Lambda_{p+1}$ for $\tau>1$
and the fact that $h$ is bounded below on $\Lambda_{p+1}$. Hence $h(x)\geqslant 0$.
Let us prove that $\lambda>0$. Otherwise, $h'(x)=f(x)-\lambda h(x)>\varepsilon\| x\|$,
which contradicts the fact that $\| h'\|\leqslant\varepsilon$. 

\medskip
For $\tau\in(0,1)$, we denote $h_\tau=(1-\tau)\lambda h+\tau f$. Clearly, $\| h_\tau-f\|<\varepsilon$.
If $\tau$ is small enough, $h_\tau$ also satisfies \eqref{1}.
Indeed, if we denote $m=\inf\{h(x);\,x\in\Lambda_{p+1}\}$, we have $m>h(x_0)$.
Therefore, 
$$
\inf\{h_\tau(x);\,x\in\Lambda_{p+1}\}\geqslant (1-\tau)\lambda m+\tau p>(1-\tau)\lambda h(x_0)+\tau f(x_0)
$$
whenever $\tau$ is small enough.

\medskip
We now fix $\tau$ such that $h_\tau(x_0)<\inf\{h_\tau(x);\,x\in\Lambda_{p+1}\}$,
we denote $\widehat{f}=h_\tau$, and we choose $c$ such that 
$\widehat{f}(x_0)<c<\inf\{\widehat{f}(x);\,x\in\Lambda_{p+1}\}$. 
The open slice $S=D\cap\{\widehat{f}<c\}$ is non empty, does not intersect $\Lambda_{p+1}$,
and it  is bounded, 
because if $x$ belongs to this slice, then 
$\| f-\widehat{f}\|\cdot\| x\|\geqslant(f-\widehat{f})(x)>\varepsilon\|x\|-c$,
thus $\| x\|\leqslant\frac{c}{\varepsilon-\| f-\widehat{f}\|}$. 

\medskip
By Lemma \ref{diam}, there exists $g\in X^*$, $\|g-\widehat{f}\|<\varepsilon-\| f-\widehat{f}\|$ and $d\in\R$ 
such that the non empty slice $C:=D\cap \{g <  d\}$ is contained
in $S$ (hence does not intersect $\Lambda_{p+1}$), and $diam\:( C )< \eta$. 
Clearly, $\| f-g\|\leqslant\| f-\widehat{f}\|+\|\widehat{f}-g\|<\varepsilon$.

\end{proof}

From now on, we fix $p\in\mathbb{Z}$.
The following result gives the existence of a ``slicing'' of $\Lambda_p\backslash\Lambda_{p+1}$
into small pieces.

\begin{lem}\label{diamCalpha}
Let $\eta\!>\!0$.
Then, there exists transfinite sequences $(f_\alpha)\in Y^*$ with $\|f_\alpha-f\|<\varepsilon$, and $(c_\alpha)$ in 
$\R$, such that, if  $(D_{\alpha})_{\alpha \leqslant \mu}$ is the transfinite
decreasing sequence of closed convex sets defined as follows :
\begin{itemize}
\item $D_0 = \Lambda_p$;
\item  for all $\alpha$,\: $D_{\alpha+1} =D_{\alpha}\backslash\{f _{\alpha}<c_{\alpha}\}$
\item $D_\alpha=\bigcap_{\gamma<\alpha}D_\gamma$ for all limit ordinal $\alpha$,
\end{itemize} 
and if, for all $\alpha$, 
$C_\alpha=D_{\alpha} \backslash D_{\alpha + 1}$, then $C_\alpha$ is non empty, 
$diam\:( C_{\alpha} )< \eta$,  $D_{\mu} =\Lambda_{p+1}$,
and $\{C_\alpha;\,\alpha<\mu\}$ is a partition of $\Lambda_p\backslash\Lambda_{p+1}$.
\end{lem}

\begin{proof}
We prove the existence of $f_\alpha$, $c_\alpha$ by transfinite induction.
Let us assume that $f_\beta$ and $c_\beta$ have been 
constructed for $\beta<\alpha$. Hence, we have constructed
$D_{\alpha}=\Lambda_p\cap\bigl(\bigcap_{\gamma<\alpha}\{ f_{\gamma}\geqslant c_{\gamma}\}\bigr)$. 
If $D_\alpha=\Lambda_{p+1}$, then we set $\mu=\alpha$ and we stop. Otherwise, we apply 
Lemma \ref{diamC} with $D=D_\alpha$   to construct
$g=f_\alpha$ and $d=c_\alpha$ such that, if $C_\alpha =D_{\alpha}\cap \{f_\alpha<c_\alpha\}$, then  
$C_\alpha$ is non empty and has diameter less than $\eta$.
Moreover, since $C_\alpha\subset\Lambda_p$ and $C_\alpha\cap\Lambda_{p+1}=\emptyset$, we have
that the union of the $C_\alpha$ is included in $\Lambda_p\backslash\Lambda_{p+1}$.
The sets $C_\alpha$, $\alpha<\mu$ are pairwise disjoints, and their union is equal to 
$\Lambda_p\backslash\Lambda_{p+1}$ because $D_\mu=\Lambda_{p+1}$, thus
$\{C_\alpha;\,\alpha<\mu\}$ is a partition of $\Lambda_p\backslash\Lambda_{p+1}$.
\end{proof}

\vskip 0.5cm
We now define a mapping $t_0$ on $\Lambda_p\backslash\Lambda_{p+1}$.

\begin{prop}\label{etacauchy}
Let  $\eta>0$. There exists a mapping
$t_0: \Lambda_p\backslash\Lambda_{p+1} \to \Sx \cap B(f,\varepsilon)$ such that 
\begin{itemize}
\item if $x\in\Lambda_p$ and $y\in X$ satisfy $\langle t_0(x), y - x \rangle \leqslant 0$,
then $y\notin\Lambda_{p+1}$,
\item for all sequence $(x_n)$ in 
$\Lambda_p\backslash\Lambda_{p+1}$, 
if   $\langle t_0(x_n), x_{n+1} - x_n \rangle \leqslant 0$ for all $n \in \N$, then $(x_n)$ is $\eta$-Cauchy.
\end{itemize}
\end{prop}

A mapping $t_0$ with the property of Proposition \ref{etacauchy} will be called 
later on an $\eta$-winning tactic (player $A$ can force the sequence $(x_n)$ to be $\eta$-Cauchy).
\

\begin{proof}
Let us first define $t_0$. 
First observe that if for $0<\varepsilon<1/2$, we have a mapping 
$t_0: \Lambda_p\backslash\Lambda_{p+1} \to B(f,\varepsilon)$, 
then the function defined by $t_1(x)=t_0(x)/\|t_0(x)\|$ has its values in $\Sx\cap B(f,2\varepsilon)$ and 
$\langle t_1(x_n), x_{n+1} - x_n \rangle \leqslant 0$ is equivalent to $\langle t_0(x_n), x_{n+1} - x_n \rangle \leqslant 0$.
So it is enough to construct $t_0: \Lambda_p\backslash\Lambda_{p+1} \to B(f,\varepsilon)$ satisfying the conclusion of Proposition \ref{etacauchy}. Let $f_\alpha$ be the functionals constructed in Lemma \ref{diamCalpha}.
For each $\alpha$, we have $\|f-f_\alpha\| < \varepsilon$.
If $x \in \Lambda_{p} \backslash \Lambda_{p+1}$, 
then there exist $\alpha$ such that $x \in C_{\alpha}$, and we set $t_0(x)=f_\alpha$.

Let us notice that if $x\in C_{\alpha}$, $y\in X$, and 
$\langle t_0(x), y - x \rangle \leqslant 0$, then 
$f_\alpha(y)\leqslant f_\alpha(x)<c_\alpha$
and the above inequality implies $y\notin D_{\alpha+1}$, and in particular $y\notin\Lambda_{p+1}$. 

Let now $(x_n)$ be a sequence in $\Lambda_p\backslash\Lambda_{p+1}$ such that 
$\langle t_0(x_n), x_{n+1} - x_n \rangle \leqslant 0$ for all $n \in \N$.
Let $\alpha_n$ be such that $x_n \in C_{\alpha_{n}}$. Since $t_0(x_n)=f_{\alpha_n}$ and
$\langle t_0(x_n), x_{n+1} - x_n \rangle \leqslant 0$, we obtain $f_{\alpha_n}(x_{n+1}) \leqslant f_{\alpha_n}(x_{n})$.
This implies that $x_{n+1} \notin D_{{\alpha_n}+1}$. But $x_{n+1} \in C_{\alpha_{n+1}}$, so $\alpha_{n+1}\leqslant \alpha_{n}$. 
Thus $(\alpha_{n})$ is a nonincreasing sequence. 
The set $A = \{\alpha_n; n\in \N\} \subset [0, \mu]$ is  well ordered, so there exists $n_0$ such that $\alpha_{n_0} = \min A$. 
Then for all $n \geqslant n_0$,  $\alpha_n = \alpha_0$. 
Now, for all $n,m \geqslant n_0$, we have $x_n, x_m \in C_{\alpha_{n_0}}$, so $\|x_n - x_m\| < \eta$. 
Thus  the sequence $(x_n)$ is $\eta$-Cauchy. 
\end{proof}

\section{Multi-$\varepsilon$-tactics.}

Whenever $E$ is a set, we denote $\mathcal{P}(E)$ the set of subsets of $E$.

\begin{defn}
Let $T: A \subset X \to \mathcal{P}(X^{*})$. We say that $t: A \to X^{*}$ is a selection of $T$ if  $t(x) \in T(x)$
for all $x \in A$.
\end{defn}

\medskip
In Lemma \ref{diamCalpha}, we have constructed $f_\alpha\in X^*$, $c_\alpha\in\R$,
$D_\alpha\subset X$ such that, if 
$C_\alpha=D_\alpha\backslash D_{\alpha+1}=D_\alpha\cap\{f_\alpha< c_\alpha\}$, then 
$\{C_\alpha;\,\alpha<\mu\}$ is a partition  of $\Lambda_p\backslash\Lambda_{p+1}$.
We now define $\displaystyle R(x)= \frac{c_\alpha-f_\alpha(x)}{4 \max\{\|u\|; u \in C_\alpha\}}$
 whenever $x \in C_{\alpha}$.

\begin{prop}\label{etatactic}
Under the notations of Lemma \ref{diamCalpha}, let us define 
$T:\Lambda_p\backslash\Lambda_{p+1}\to\mathcal{P}\bigl(\Sx\cap B(f,\varepsilon)\bigr)$ by 
$T(x)=\Sx\cap B(f,\varepsilon)\cap\overline{B}(f_\alpha,R(x))$ whenever $x \in C_{\alpha}$.
Then, for each selection $t$ of $T$, $t$ is an $\eta$-winning tactic.
\end{prop}

\begin{proof}
Let $t$ be a selection of $T$, and let us prove that the selection $t$ is $\eta$ winning.
If $x \in C_{\alpha}$ and $t(x)(y) \leqslant t(x)(x)$ then, according to 
Lemma \ref{radius}, $y \notin D_{\alpha + 1}$, and in particular,
$y\notin\Lambda_{p+1}$. Let now $(x_n)$ be a sequence in 
$\Lambda_{p} \backslash\Lambda_{p+1}$ such that 
$\langle t(x_n), x_{n+1} - x_n \rangle \leqslant 0$ for all $n \in \N$.
Let $\alpha_n$ be such that $x_n \in C_{\alpha_{n}}$. Since 
$t(x_n)(x_{n+1})  \leqslant t(x_{n})(x_n)$ and $x_n \in C_{\alpha_{n}}$, 
we obtain that $x_{n+1}\notin D_{\alpha_n + 1}$.
But $x_{n+1} \in C_{\alpha_{n+1}}$, so $\alpha_{n+1}\leqslant\alpha_n$,
hence $(\alpha_{n})$ is a non increasing sequence of ordinals. 
Therefore the sequence $(\alpha_{n})$
is stationary, and, as in the proof of Proposition \ref{etacauchy},  
all the $x_n$ except finitely many of them are in the same $C_\alpha$ which has diameter less than $\eta$. 
Thus,  $(x_n)$ is $\eta$-Cauchy.\\
\end{proof}

\section{A sequence of multi-$\varepsilon$-tactics.}

\begin{lema}\label{beta} Assume that $X$ has the Radon-Nikodym property.
Let $\eta,r>0$ and $D$ be a closed convex set of $X$. Let $\widehat{f}\in X^*$ and $c\in\R$ be such that $S = D \cap \{\widehat{f} < c\}$ is a non empty bounded set. 
Then, there exists  transfinite sequences $(g_\beta)_{1\leqslant\beta<\mu}$ in $\Sx$ 
and $(d_\beta)_{1\leqslant\beta<\mu}$ 
in $\R$ such that, if  $(D_{\beta})_{0\leqslant\beta \leqslant \mu}$ is defined as follows :
$$
\text{for all }\beta\geqslant 0, \,\, D_\beta=D\cap\bigl(\bigcap_{\gamma<\beta} \{g_{\gamma}\geqslant d_{\gamma}\}\bigr)
$$
Then, for all $\beta<\mu$, $D_\beta\supset D\cap\{\widehat{f}\geqslant c\}$, and, 
 if we denote $C_\beta=D_\beta\backslash D_{\beta+1}$, we have
\begin{itemize}
\item[(i)] $C_{\beta}\ne\emptyset$ and $diam\:( C_{\beta})< \eta$. 
\item[(ii)] $\|g_{\beta} - \widehat{f}\| < \min\big\{r, \inf\{R(u); u \in C_{\beta}\}\big\}$.
\item[(iii)] $\big\{C_\beta;\,\beta<\mu\big\}$ is a partition of $S$.
\end{itemize}
\end{lema}

\begin{proof}
We shall construct $g_\beta$ and $d_\beta$ by transfinite induction using Lemma \ref{diam} at each step.
Let us assume that $g_\gamma$ and $d_\gamma$ have been 
constructed for $\gamma<\beta$. 
Hence $D_\beta=D\cap\bigl(\bigcap_{\gamma<\beta}\{g_\gamma\geqslant d_\gamma\}\bigr)$ 
is well defined (notice that $D_0=D$). 

If $D_\beta\cap \{\widehat{f}<c\}$ is non empty, it is also bounded because it is included in 
$S=D\cap \{\widehat{f}<c\}$.  Applying Lemma \ref{diam} with $D_\beta$ in place of $D$, we find
$g_\beta$ and $d_\beta$ such that $C_{\beta} = D_{\beta}\cap\{g_{\beta} <  d_{\beta}\}$ 
satisfies conditions $(i)$, $(ii)$
and $C_{\beta}\subset S$. This last condition implies that 
$D_{\beta+1}=D_\beta\backslash C_\beta\supset D\cap\{\widehat{f}\geqslant c\}$.

If $D_\beta=D\cap\{\widehat{f}\geqslant c\}$, then we set $\mu=\beta$ and we stop
and condition $(iii)$ is satisfied.
\end{proof}

We are now ready to  construct a decreasing sequence $(T_k)$ of multi-$\varepsilon$-tactics.

\noindent
\begin{thm}\label{sequence}
Le us fix a sequence $(\eta_k)$ converging to $0$ such that $\eta_k > 0$ for all $k$. There exists a sequence $(T_k)$ of multivalued functions from  $\Lambda_{p} \backslash\Lambda_{p+1}$ to $\Sx$ such that 
$T_k(x) = \Sx \cap \overline{B}(\widehat{f}_{x,k} , r_{k}(x))$, where $\widehat{f}_{x,k}\in\Sx$, $r_{k}(x) > 0$, $r_{k}(x)\to 0$ and 
$T_{k+1}(x) \subset T_{k}(x)$, and with the property that, 
for all $t$ selection of $T_k$,  $t$ is an $\eta_{k}$-winning tactic.
\end{thm}

\begin{proof}
The construction will be carried out by induction on $k$.
\\ \\
\bf Construction of $T_0$. 

\noindent\rm It is enough to apply Proposition \ref{etatactic} with $\eta=\eta_0$.
\\ \\
\bf Induction step. \rm

\noindent
Assume $T_{k}(x) =  \Sx \cap \overline{B}(\widehat{f}_{x,k}, r_{k}(x))$ has been constructed with the following properties~: 
\begin{itemize}
\item[-] There exists a partition of $\Lambda_p \backslash \Lambda_{p+1}$, 
 given by $C_{\alpha, k} = D_{\alpha, k} \cap \{f_{\alpha, k} < c_{\alpha,k}\}$ with $\alpha<\mu_k$,
 such that $diam (C_{\alpha, k}) < \eta_{k}$
and $\widehat{f}_{x,k} = f_{\alpha, k}$ whenever $x\in C_{\alpha, k}$. 
\item[-]  If $x\in C_{\alpha,k}$, $r_{k}(x) = \min\{R_{k}(x), r_{k}\}$, 
 where $r_k>0$ is constant on $C_{\alpha, k}$ and 
 $\displaystyle R_k(x)= \frac{c_{\alpha,k}-f_{\alpha,k}(x)}{4 \max\{\|u\|; u \in C_{\alpha,k}\}}$.
 \item[-]  For all $t$ selection of $T_k$,  $t$ is an $\eta_{k}$-winning tactic.
 \end{itemize} 
 
 \noindent
 Since $\big\{C_{\alpha,k};\,\alpha<\mu_k\big\}$ is a partition of $\Lambda_p \backslash \Lambda_{p+1}$,
 it is enough, for each $\alpha<\mu_k$, to define $T_{k+1}$ on $C_{\alpha,k}$.
Using Lemma \ref{beta} with $D=D_{\alpha,k}$, 
$\widehat{f}=f_{\alpha,k}$ and $c=c_{\alpha,k}$, there exists  
$g_{\alpha,\beta}\in\Sx$ and $d_{\alpha,\beta}\in\R$ for $\beta<\mu_{\alpha,k}$,
such that 
$\|g_{\alpha,\beta}-\widehat{f}_{x,k}\|<r_{k}(x)$, and, if 
$$
D_{\alpha,\beta}=D_{\alpha,k}\cap\bigl(\bigcap_{\beta<\mu_{\alpha,k}} \{g_{\alpha,\beta}\geqslant d_{\alpha,\beta}\}\bigr),
$$ 
then $D_{\alpha,\beta+1}\supset D_{\alpha+1,k}$, $C_{\alpha,\beta}=D_{\alpha,\beta}\backslash D_{\alpha,\beta+1}$ is non empty,
have diameter less than $\eta_{k+1}$ and $\big\{C_{\alpha,\beta};\,\beta<\mu_{\alpha,k}\big\}$  
is a partition of $S=D_{\alpha,k}\cap\{f_{\alpha, k} < c_{\alpha,k}\}=C_{\alpha, k}$. 
For each $x \in C_{\alpha,\beta}$, we denote 
$$
\widehat{f}_{x,k+1}=g_{\alpha,\beta}\quad\text{and}\quad 
r_{k+1}=\min\big\{r_{k}, \inf\{R_k(u);\,u\in C_{\alpha,\beta}\}\big\} - \|g_{\alpha,\beta} - f_{\alpha, k}\|>0.
$$
$R_{k+1}(x)$ is then defined by 
$\displaystyle R_{k+1}(x)= \frac{d_{\alpha,\beta}-g_{\alpha,\beta}(x)}{4 \max\{\|u\|; u \in C_{\alpha,\beta}\}}$.
Therefore, we have defined
$r_{k+1}(x)=\min\big\{R_{k+1}(x),r_{k+1}\big\}$
and $T_{k+1}(x)=\Sx\cap \overline{B}(\widehat{f}_{x,k+1}, r_{k+1}(x))$.
We claim that $T_{k+1}(x)\subset T_{k}(x)$. Indeed, for $x \in C_{\alpha,\beta}$,
$$
T_{k+1}(x)\subset 
\overline{B}(g_{\alpha,\beta},r_{k+1})\subset
\overline{B}(f_{\alpha,k},\|f_{\alpha,k}-g_{\alpha,\beta}\|+r_{k+1})
\subset\overline{B}(\widehat{f}_{x,k}, r_{k}),
$$
and, on the other hand,
$$
T_{k+1}(x)\subset\overline{B}(g_{\alpha,\beta},r_{k+1})
\subset \overline{B}(g_{\alpha,\beta},R_{k}(x) - \|g_{\alpha,\beta}- f_{\alpha, k}\|)
\subset \overline{B}(\widehat{f}_{x,k},R_k(x)).
$$
If $x\in C_{\alpha,\beta}$ and $g\in T_{k+1}(x)$,
since $\|g-g_{\alpha,\beta}\|\leqslant R_{k+1}(x)$, we can apply Lemma \ref{radius}
with $D=D_{\alpha,\beta}$, $\widehat{f}=g_{\alpha,\beta}$ and $c=d_{\alpha,\beta}$ to obtain
$\{g\leqslant g(x)\}\cap D_{\alpha,\beta+1}=\emptyset$, and since
$D_{\alpha,\beta+1}\supset D_{\alpha+1,k}$, we also have
$\{g\leqslant g(x)\}\cap D_{\alpha+1,k}=\emptyset$. 
Thus,  if $y\in X$ and $g(y)\leqslant g(x)$ then $y\notin \Lambda_{p+1}\subset D_{\alpha+1,k}$.
Also, if $y\in\Lambda_p$ and $g(y)\leqslant g(x)$,  then either 
$y \in C_{\alpha,\beta'}$ with $\beta'\leqslant\beta$
or $y\in C_{\alpha'}$ for some $\alpha'\leqslant\alpha$.

\medskip
The set $E=\big\{(\alpha,\beta);\,\alpha<\mu_k,\,\beta<\mu_{\alpha,k}\big\}$
is well ordered by the relation $(\alpha,\beta)\le(\alpha',\beta')$
if and only if  either  $\alpha=\alpha'$ and $\beta\leqslant\beta'$, or $\alpha\leqslant\alpha'$.
So there exists a unique ordinal $\mu_{k+1}$ and an order preserving bijection 
from $\pi:[0,\mu_{k+1})$ onto $E$. We then define, for $\alpha<\mu_{k+1}$, 
$C_{\alpha, k+1}=C_{\pi(\alpha)}$, $f_{\alpha,k+1}=g_{\pi(\alpha)}$
and $c_{\alpha,k+1}=d_{\pi(\alpha)}$.
Therefore $\big\{C_{\alpha, k+1};\,\alpha \leq \mu_{k+1}\big\}$ is 
a partition of $\Lambda_p \backslash \Lambda_{p+1}$
into sets of diameter less than $\eta_{k+1}$. Moreover, if 
$x\in C_{\alpha,k+1}$ and $g\in T_{k+1}(x)$, then for all $y\in\Lambda_p\backslash\Lambda_{p+1}\cap\{g\leqslant g(x)\}$,
there exists $\alpha'\leqslant\alpha$ such that $y\in C_{\alpha',k+1}$. 

\medskip
Let us now prove that, if $t$ be a selection of $T_{k+1}$, then $t$ is $\eta_{k+1}$-winning. 
If $x \in C_{\alpha,k+1}$ and $y\in X$ satisfy
$t(x)(y) \leqslant t(x)(x)$, then $y\notin\Lambda_{p+1}$, and in the case $y\in\Lambda_p$, then 
$y\in C_{\alpha',k+1}$ for some $\alpha'\leqslant\alpha$.
Let now $(x_n)$ be a sequence in $\Lambda_{p} \backslash\Lambda_{p+1}$ such that 
$\langle t(x_n), x_{n+1} - x_n \rangle \leqslant 0$ for all $n \in \N$.
Let $\alpha_n$ be such that $x_n \in C_{\alpha_{n},k+1}$. Since 
$t(x_n)(x_{n+1})  \leqslant t(x_{n})(x_n)$, 
we obtain  that 
$\alpha_{n+1}\leqslant\alpha_n$. Thus $(\alpha_{n})$ is a non increasing sequence of ordinals, 
hence there exists $n_0$ such that,
for all $n\geqslant n_0$, $\alpha_n=\alpha_{n_0}$, 
All the $x_n$, except finitely many of them, are in  
$C_{\alpha_{n_0},k+1}$ which has diameter less than $\eta_{k+1}$. 
This proves that the sequence $(x_n)$ is $\eta_{k+1}$-Cauchy.
This completes the induction.

\end{proof}

\section{Proof of Theorem \ref{converges}}

\begin{proof}
For each  $p\in\mathbb{Z}$, we define $t(x)$ whenever $x \in \Lambda_{p} \backslash \Lambda_{p+1}$. In this case, $\bigl(T_k(x)\bigr)$ is a decreasing sequence of closed sets
in the Banach space $X^*$ and $diam\bigl(T_k(x)\bigr)\to 0$. Therefore $\bigcap T_k(x)$ is a singleton, and we denote $t(x)$ the unique element
of this intersection. Whenever $x\in\Lambda_p$, we have $t(x)\in T_1(x)$, so 
$$
x\in\Lambda_p\quad\text{and}\quad\langle t(x), y - x \rangle \geqslant 0\quad\Rightarrow\quad y \notin \Lambda_{p+1}
$$
Let us prove that $t$ is a winning tactic in $\Lambda_{p} \backslash\Lambda_{p+1}$. Let us fix a sequence $(x_n)\in \Lambda_{p} \backslash \Lambda_{p+1}$
such that for each $n$, $\langle t(x_n), x_{n+1} - x_n \rangle \leqslant 0$.
Since $t(x)\in T_k(x)$, the sequence is $\eta_k$-Cauchy. Since this is true for all $k\in\N$, the sequence $(x_n)$ converges.
Now let $(x_n)$ be a sequence such that 
the sequence $\bigl(f(x_n) - \varepsilon \|x_n\|\bigr)$ is bounded below and  $\langle t(x_n), x_{n+1} - x_n \rangle \leqslant 0$ for all $n \in \N$.
For each $n$, there exists an integer $p_n\in\mathbb{Z}$ such that
$x_n \in \Lambda_{p_n}\backslash\Lambda_{p_{n + 1}}$. 
Since $\langle t(x_n), x_{n+1} - x_n \rangle \leqslant 0$, $x_{n+1}\notin\Lambda_{p_n+1}$, so $p_{n+1}\leqslant p_n$. 
Since $\bigl(f(x_n) - \varepsilon \|x_n\|\bigr)$ is bounded below, the sequence $(p_n)$ is bounded below.
Thus $(p_n)$ is a nonincreasing sequence which is bounded below, therefore there exists $n_1$ such that 
$p_n = p_{n_1} : = p$ for all $n\leqslant n_1$. So, the whole sequence 
$(x_n)_{n\geqslant n_1}$ is included in $\Lambda_p\backslash\Lambda_{p + 1}$. 
Since $t |_{\Lambda_{p} \backslash \Lambda_{p + 1}}$ is a winning tactic in  $\Lambda_{p}\backslash \Lambda_{p + 1}$
and $\langle t(x_n), x_{n+1} - x_n \rangle \leqslant 0$,
the sequence $(x_n)$ is convergent.

\end{proof}


\end{document}